\documentclass{amsart}
\usepackage{tikz,tikz-cd}
\usepackage{url}
\usepackage{graphicx}
\usepackage{graphicx}
\usepackage{color}
\usepackage{times}
\usepackage{cite}
\usepackage{enumitem,latexsym}
\usepackage{latexsym}
\usepackage{amsmath,amssymb}
\usepackage{graphicx}
\usepackage{amsthm}
\usepackage{verbatim}

\newcommand{\measrestrict}{%
  \,\raisebox{-.127ex}{\reflectbox{\rotatebox[origin=br]{-90}{$\lnot$}}}\,%
}
\newtheorem{thm}{Theorem}[section]
\newtheorem*{theorem*}{Theorem}
\newtheorem*{acknowledgement*}{Acknowledgement}
\newtheorem{cor}[thm]{Corollary}
\newtheorem{claim}[thm]{Claim}

\newtheorem{lem}[thm]{Lemma}
\newtheorem{prop}[thm]{Proposition}
\theoremstyle{definition}
\newtheorem{defn}[thm]{Definition}
\theoremstyle{remark}
\newtheorem{rem}[thm]{Remark}

\numberwithin{equation}{section}

\newcommand{\set}[1]{\left\{#1\right\}}
\newcommand{\Real}{\mathbb R}

\newcommand{\func}[1]{\ensuremath{\mathop{\mathrm{#1}}} }

\newcommand{\dist}[0]{\mathrm{dist}}

\newcommand{\spt}[0]{\func{spt}}

\begin{document}

\title[The level set flow of a low entropy hypersurface does not disconnect]{The level set flow of a hypersurface in $\Real^4$ of low entropy does not disconnect}
\date{}
\author{Jacob Bernstein}
\address{
Department of Mathematics\\
Johns Hopkins University\\
Baltimore, MD 21218, USA}
\email{bernstein@math.jhu.edu }

\author{Shengwen Wang}
\address{
Department of Mathematics\\
Binghamton University\\
Vestal, NY 13850, USA}
\email{swang@math.binghamton.edu }

\thanks{The first author was partially supported by the NSF Grant DMS-1609340.}


\begin{abstract}
We show that if $\Sigma\subset \Real^4$ is a closed, connected hypersurface with entropy $\lambda(\Sigma)\leq \lambda(\mathbb{S}^2\times \Real)$, then the level set flow of $\Sigma$ never disconnects.  We also obtain a sharp version of the forward clearing out lemma for non-fattening flows in $\Real^4$ of low entropy.
\end{abstract}
\maketitle

\section{Introduction}
A family of hypersurfaces $\Sigma_t\subset\mathbb R^{n+1}$ evolves by mean curvature flow (MCF) if it satisfies
\begin{equation}\label{MCF}
\left(\frac{\partial}{\partial t}\mathbf{x}_{\Sigma_t}\right)^\perp=\mathbf{H}_{\Sigma_t}
\end{equation}
here a hypersurface is a smooth submanifold of codimension one and $\mathbf{x}_{\Sigma_t}$ is the position vector, $\mathbf{H}_{\Sigma_t}$ is the mean curvature vector and $\perp$ is the projection onto the normal of $\Sigma_t$.  A fundamental property of MCF is that the flow of a closed hypersurface must develop a singularity in finite time.  If one considers the level set flow (see Chen-Giga-Goto \cite{CGG} and Evans-Spruck \cite{ES1,ES2,ES3,ES4}), then one obtains a canonical set theoretic weak mean curvature flow that persists through singularities and, for closed initial data, vanishes in finite time. By definition, as long as the flow is smooth, then the topology of the hypersurfaces does not change, however this need not be the case for the level set flow after the first singularity.

When $n=1$, it follows from  Gage-Hamilton \cite{GH} and Grayson \cite{Grayson} that the flow disappears when it becomes singular.  In particular, the flow remains connected until it disappears. In contrast, when $n>1$, non-degenerate neck-pinch examples show that there are flows that become singular without disappearing.  In these examples,  the level set flow disconnects after the neck-pinch singularity. In \cite{BW3}, the first author and L. Wang showed that, when $n=2$ and the initial entropy is small (see \eqref{EntropyDef} below),  then the flow disappears at its first singularity.  This result makes use of a classification of singularity models in $\Real^3$ of low entropy from \cite{BW3} and whether such a classification exists in higher dimension is unknown.  In the present note we show that when $n=3$ and the initial hypersurface is closed, connected and of low entropy, then even if the flow forms a singularity before it disappears, its level set flow remains connected until its extinction time.

\begin{thm}\label{main}
Let $\Sigma\subset\mathbb R^{4}$ be a closed, connected hypersurface and let $\{\Gamma_t\}_{t\in[0,T]}$ be the level set flow with initial condition $\Gamma_0=\Sigma$ and extinction time $T$.  If $\lambda(\Sigma)\leq \lambda(\mathbb{S}^{2}\times\mathbb R)$, then, for all $t\in[0,T]$, $\Gamma_t$ is connected.  Moreover, if $W[t]=\mathbb R^{4}\setminus\Gamma_t$, then $W[t]$ has at most two connected components for all $t\in[0,T]$. 
\end{thm}
A technical feature of the level set flow is that it may ``fatten", i.e., develop non-empty interior.  If this occurs in Theorem \ref{main}, then there will be a $T_0\in [0, T)$ so that $W[t]$ has two components for $t\in [0, T_0)$ and one component for $t\in [T_0, T]$ -- see Theorem \ref{RefinedMainThm} for a proof of this fact. Roughly speaking, the idea is that if the flow fattens, then there are two natural flows starting from $\Sigma$, the innermost flow and the outermost flow and the level set flow lies between these two flows. In this situation, $T_0$ is the extinction time of the inner flow and $T$ is the extinction time of the outer flow.

In \cite{W}, the second author showed, for flows of low entropy, a forward in time analog of the clearing out lemma. Specifically he showed that if such a flow reaches the point $x_0$ at time $t_0$, then the flow remains near $x_0$ after $t_0$ until it disappears.   This is a forward in time analog of the standard, unconditional, clearing out lemma -- e.g., \cite[Theorem 3.1]{ES3} -- that says that if the flow reaches $x_0$ at time $t_0$, then the flow must be near $x_0$ at earlier times.   Theorem \ref{main} allows us to sharpen the result from \cite{W} and prove the forward clearing out lemma in $\Real^4$ with the optimal upper bound on the entropy.
\begin{cor}\label{forward}
Given $\epsilon>0$, there exist uniform constants $C=C(\epsilon)>0$ and $\eta=\eta(\epsilon)>0$, so that if $\{M_t\}_{t\in [0, T]}$ is a a non-fattening level set flow in $\Real^4$ that starts from a smooth closed hypersurface $M_0\subset \mathbb R^{4}$ with $\lambda(M_0)\leq \lambda(\mathbb S^{2}\times \Real)-\epsilon$,  $x_0\in M_{t_0}$ and $M_{t_0+R^2}\neq\emptyset$, then for all  $\rho\in (0, \frac{R}{2C}),$
\begin{equation*}
\mathcal H^3(B_{\rho}(x_0)\cap M_{t_0+C^2\rho^2})\geq\eta\rho^3.
\end{equation*}
Here $\mathcal H^3$ denotes three-dimensional Hausdorff measure.
\end{cor}
\begin{rem}
The entropy assumption can be seen to be sharp by considering the translating bowl soliton in $\Real^4$ of larger and larger speed, and, in the closed setting, by considering a sequence of unit spheres at increasing distance from one another and joined by a thin tube.
\end{rem}

\section{Notation and Backgound}
Let $B_R(x_0)$ be the open ball in $\Real^{n+1}$ centered at $x_0$ and, for a set $K\subset \Real^{n+1}$, let $$ T_r(K)=\bigcup_{x\in K} B_{r}(x)$$
be the $r$-tubular neighborhood of $K$.  For any $\rho>0$, $x_0\in\mathbb R^{n+1}$ and subset $\Omega\subset\mathbb R^{n+1}$, set
\begin{equation*}
\begin{split}
\Omega+x_0&=\{x+x_0\in\mathbb R^{n+1}:x\in\Omega\}\mbox{ and }
\rho\Omega=\{\rho x:x\in\Omega\}.
\end{split}
\end{equation*}
Following \cite{CM}, the entropy of a closed hypersurface, $\Sigma$, is defined by
\begin{equation}\label{EntropyDef}
\lambda(\Sigma)=\sup_{(\mathbf y,\rho)\in\mathbb R^{n+1}\times\mathbb R}  F(\rho\Sigma+\mathbf y)
\end{equation}
where $F$ is the Gaussian area of $\Sigma$ given by
\begin{equation}
F(\Sigma) = (4\pi)^{-\frac{n}{2}}\int_\Sigma e^{-\frac{|x|^2}  {4}  } d\mathcal{H}^n.
\end{equation}
The entropy and Gaussian area readily extend to the less regular objects studied in geometric measure theory.  Clearly, $\lambda(\Real^n)=1$.
If $\mathbb{S}^n$ is the unit $n$-sphere in $\Real^{n+1}$, then
\begin{equation*}
\Lambda_k=\lambda(\mathbb S^k)=\lambda(\mathbb S^k\times\mathbb R^{n-k})=F(\sqrt{2k}\mathbb{S}^k)
\end{equation*}
and so, by a computation of Stone \cite{St},
\begin{equation}
2>\Lambda_1>\frac{3}{2}>\Lambda_2>\ldots>\Lambda_n>\ldots\rightarrow\sqrt2.
\end{equation}

Let us now briefly recall some background results in the theory of (weak) mean curvature flow -- our primary sources are \cite{ES1, ES2, ES3, ES4} and \cite{I1}.  We begin with the level set flow, whose mathematical theory  was developed by Chen-Giga-Goto \cite{CGG} and Evans-Spruck \cite{ES1,ES2,ES3,ES4}. 

Let $\Gamma$ be a non-empty compact subset of $\mathbb R^{n+1}$. Select a Lipschitz function $u_0$ so that $\Gamma=\{x:u_0(x)=0\}$ and so that
$u_0(x)=-C\mbox{ when } |x|\geq R$
for some constants $C,R>0$.
For such a $u_0$,  $\{u_0\geq a>-C\}$ is compact. In \cite{ES1}, Evans-Spruck established the existence and uniqueness of viscosity solutions to the initial value problem:
\begin{equation}\label{levelsetflow}
\begin{cases}
u_t=\Sigma_{i,j=1}^{n+1}(\delta_{ij}-u_{x_i}u_{x_j}|Du|^{-2})u_{x_ix_j}\:\:\:\mbox{ on }\:\mathbb R^{n+1}\times(0,\infty)\\
u=u_0\:\:\:\mbox{ on }\:\mathbb R^{n+1}\times\{0\}.
\end{cases}
\end{equation}
Setting $\Gamma_t=\{x:u(x,t)=0\}$, define $\{\Gamma_t\}_{t\geq0}$ to be the level set flow of $\Gamma=\Gamma_0$. As shown in \cite{ES1}, the $\Gamma_t$ depend only on $\Gamma$ and are independent of the choice of $u_0$.  The level set flow has a uniqueness property and satisfies an avoidance principle.  As such, for any closed initial set, the level set flow vanishes after a finite amount of time.  Furthermore, as long as the initial set is a closed hypersurface, the level set flow agrees with the classical solution to \eqref{MCF} as long as the latter exists. A technical feature of the level set flow is that some time slices may develop non-trivial interior -- a phenomena called ``fattening". Importantly, initial sets are generically non-fattening -- see for instance \cite[Theorem 11.3]{I1}

In addition to the level set flow, we will also need to consider the measure theoretic version of MCF introduced by Brakke. An \emph{$n$-dimensional Brakke flow} (or \emph{Brakke motion}), $\mathcal K$, in $\mathbb R^{n+1}$ is a family of Radon measures $\mathcal K = \{\mu_t\}_{t\in I}$, that satisfies \eqref{MCF} in the sense of being a negative gradient flow, see \cite{I1} for the precise definition. The Brakke flow is \emph{integral} if for almost every $t\in I$, $\mu_t\in \mathbf {IM}_n(\mathbb R^{n+1})$, that is, $\mu_t$ is an integer $n$-rectifiable Radon measure.  The Hausdorff $n$-measure, $\mathcal{H}^n$ restricted to any classical solution of \eqref{MCF} is an integral Brakke flow.

Denote the parabolic rescaling and translation of a Brakke flow $\mathcal K=\{\mu_t\}$ by
\begin{equation*}
\begin{split}
D_\rho\mathcal K&= \set{\mu^{\rho,0}_{\rho^{-2}t}}\mbox{ and }\mathcal K-(x_0,t_0)=\set{\mu^{1,x_0}_{t+t_0}}
\end{split}
\end{equation*}
where
$$
\mu^{\rho,x_0}(A)=\rho^{n}\mu(\rho^{-1} A+x_0).
$$
It follows from the Brakke's compactness theorem \cite[7.1]{I1} and the Huisken monotonicity formula \cite{H, I2} that given an integral Brakke flow $\mathcal K = \{\mu_t\}_{t\in I}$ with uniformly bounded area ratios, for any $t_0>\inf I$ and $x_0\in \Real^{n+1}$ and any sequence $\rho_i\to \infty$ there exists a subsequence $\rho_{i_j}\to \infty$ so that $D_{\rho_{i_j}}(\mathcal K-(x_0,t_0))$ converges (in the sense of Brakke flows -- see \cite{I1}) to a Brakke flow $\mathcal{T}=\set{\nu_t}_{t\in \mathbb{R}}$. We call such a flow a \emph{tangent flow} to $\mathcal{K}$ at $(x_0,t_0)$ and denote the set of all possible limits (for different sequences of scalings) by $\mathrm{Tan}_{(x_0,t_0)} \mathcal{K}$. By Huisken's monotonicity formula, $\mathcal{T}\in \mathrm{Tan}_{(x_0,t_0)} \mathcal{K}$ is backwardly self-similar. If $\nu_{-1}=\mathcal{H}^n \measrestrict \Upsilon$ for a smooth hypersurface $\Upsilon$, then $\Upsilon$ satisfies the equation
\begin{equation}
\label{SSEqn}
\mathbf H_\Upsilon + \frac{\mathbf x^\perp}{2} =0.
\end{equation}
Any hypersurface, $\Upsilon$, that satisfies \eqref{SSEqn} is called a \emph{self-shrinker} and is \emph{asymptotically conical} if $\lim_{\rho\to 0} \rho \Upsilon =C$ in $C^{\infty}_{loc}(\Real^{n+1}\backslash \set{0})$ for some regular cone $C$.  For instance, any hyperplane through the origin is an asymptotically conical self-shrinker.

A feature of Brakke flows is that they may suddenly vanish.  In order to handle technical issues that arise from this possibility we will need Ilmanen's enhanced motions \cite[8.1]{I1}\cite{W5}. Following the formulation in \cite{W5},  a pair $(\tau,\mathcal K)$ is an \emph{enhanced motion}, if $\tau\in\mathbf I_{n+1}^{loc}(\mathbb R^{n+1}\times\mathbb R)$ is a locally $(n+1)$-dimensional integral current in space-time and $\mathcal K=\{\mu_t\}_{t\in\mathbb R}$ is a Brakke flow that together  satisfy 
\begin{enumerate}
\item  $\partial \tau = 0$ and $\partial (\tau_{t\geq s})=\tau_s$ and $\tau_t\in\mathbf I_{n}(\mathbb R^{n+1})$ for each time slice $t$
\item $\partial \tau_t = 0$ for all $t$  
\item $t\mapsto \tau_t$ is continuous in the flat topology
\item $\mu_{\tau_t} \leq\mu_t$ for all $t$ 
\item $V_{\mu_t}=V_{\tau_t}+2W_t$ for some integral varifold $W_t$ for a.e. $t$. In other words, they are compatible for a.e. $t$ as defined in \cite{W5}.
\end{enumerate}
Here $\tau$ is the called the \emph{undercurrent} and $\mathcal K$ is the \emph{overflow}. Likewise $(\tau, \mathcal{K})$ is an \emph{enhanced motion} with initial condition $\tau_0\in \mathbf{I}_n(\Real^{n+1}\times\set{t_0})$ if the above holds  for all $t\geq t_0$ and $\partial \tau=\tau_0$. An enhanced motion $(\tau,\mathcal K)$ is a \emph{matching motion} if $\mu_{\tau _t}=\mu_t$ for a.e. $t$ for which this makes sense. 

%

Associated to each $E\subset \Real^{n+1}\times \Real$ of locally finite perimeter, there is a unique $(n+2)$-dimensional integral current $[E]\in \mathbf{I}_{n+2}^{loc}(\Real^{n+1}\times \Real)$. Similarly, given an  oriented dimension-$k$ submanifold $\Sigma\subset \Real^{n+1}\times \Real$ there is a unique $[\Sigma]\in \mathbf{I}_k^{loc}(\Real^{n+1}\times \Real)$.  If $\partial^*E$ is the reduced boundary of $E$, then $[\partial^* E]=\partial[E]\in \mathbf{I}_{n+1}^{loc}(\Real^{n+1}\times \Real)$.  As such, there is an integer $(n+1)$-rectifiable Radon measure $\mathcal H^n\measrestrict\partial^* E$ -- see \cite{I1} for details. In what follows when we refer to a set of finite perimeter, we mean a specific set, $E$, that has finite perimeter, not  an equivalence class of sets.  In particular, we may \emph{a priori} have $\overline{\partial^* E} \neq \partial E$.

We extend the notion of canonical boundary motion from \cite{BW2} -- see also \cite{I1, BW4}.  These flows are special cases of flows introduced by Ilmanen in \cite{I1} that synthesize the level set flow and Brakke flow in a natural way and are key to our approach. 
\begin{defn}
 A \emph{canonical boundary motion} is a triple $(E_0,E,\mathcal K)$ consisting of an open bounded set $E_0\subset \Real^{n+1}\times \set{0}$ with $\partial E_0$ a smooth closed hypersurface, an open bounded set $E\subset \mathbb R^{n+1}\times\mathbb [0, \infty)$ of finite perimeter and a Brakke flow $\mathcal K=\{\mu_t\}_{t\geq 0}$ so:
\begin{enumerate}
\item $E=\{(x,t):u(x,t)>0\}$, where $u$ solves equation (\ref{levelsetflow}) with $u_0$ chosen so $E_0=\{x:u_0(x)>0\}$ and $\partial E_0=\{x:u_0(x)=0\}$;
\item The level set flow of $\partial E_0$ is non-fattening;
\item  For $t\geq 0$, each $E_t=\{x:(x,t)\in E\}$ is of finite perimeter and $\mu_t=\mathcal H^n\measrestrict\partial^*E_t$.
\end{enumerate}
If, in addition, 
\begin{enumerate}[resume*]
\item  $\set{u=0}=\overline{\partial^* E}$  in $\Real^{n+1}\times (0,\infty)$,
\end{enumerate}
where $u$ is from Item (1), then $(E_0,E,\mathcal{K})$ is a \emph{strong canonical boundary motion}.
\end{defn}
\begin{rem} \label{StrongBndryRem} Observe,  $\set{u>0}= E\subset\bar{E}\subset \set{u\geq 0}$ for a canonical boundary motion and $\bar{E}= \set{u\geq 0}$ for a strong canonical boundary motion.  If $\Gamma_t=\{x\in \Real^{n+1}| u(x,t)=0\}$, then $\set{\Gamma_t}_{t\geq 0}$ is the level set flow of $\Gamma_0=\Sigma$ and is non-fattening. Clearly, $\partial E_t\subset \Gamma_t$, but equality need not hold -- even for strong canonical boundary motions.  For instance, at the extinction time of any compact flow one does not have equality. 
\end{rem}
By \cite[11.4]{I1}, for a $E_0$ with the property that the level set flow of $\partial E_0$ is non-fattening, there are $E$ and $\mathcal{K}$ so $(E_0,E,\mathcal{K})$ is a canonical boundary motion.  In general, the non-fattening condition is not enough to ensure the existence of a strong canonical boundary motion, however, in \cite[12.11]{I1},  Ilmanen shows such existence for ``generic" $E_0$.
%
%


Finally, we introduce the following notation for a level set flow $\{\Gamma_t\}_{t\geq 0}$ in $\Real^{n+1}$, $n\geq 1$,
\begin{equation*}
\begin{split}
W[t]&=\mathbb R^{n+1}\setminus\Gamma_t\\
W[s,r]&=\{(x,t)|x\in(\mathbb R^{n+1}\setminus\Gamma_t),s\leq t\leq r\}=\bigcup_{t\in [s,r]} W[t]\\
n(t)&=\#\{\text{connected components of $W[t]$}\}\in \mathbb{N}\cup\set{\infty}.
\end{split}
\end{equation*}
As $\Gamma_t$ is compact and $n\geq 1$, there is exactly one unbounded component of $W[t]$, denoted by $W^-[t]$. Let $W^+[t]=W[t]\backslash W^-[t]$ be the bounded components and set
$$
W^\pm[s,r]=  \bigcup_{t\in [s,r]} W^\pm [t]. 
$$

\section{Proof of Theorem \ref{main} for strong canonical boundary motions}
In this section we show Theorem \ref{main} for flows that are strong canonical boundary motions. We begin with several preliminary results. 
The first is an elementary topological result -- we include a proof for the sake of completeness.
\begin{lem}\label{TopLem}
Let $\Gamma\subset \Real^{n+1}$ be a compact set.  If  $\Real^{n+1}\backslash \Gamma$ has exactly two components, $W^\pm$, and $\Gamma=\partial W^\pm$, then $\Gamma$ is connected.
\end{lem}
\begin{proof}
	Suppose that $\Gamma$ is not connected.  Let $K$ be one component of $\Gamma$ and $K'=\Gamma\backslash K\neq \emptyset$.  Observe that both $K$ and $K'$ are compact and so there is a $r>0$ so that $T_r(K)\cap T_r(K')=\emptyset$ and, hence, $T_r(\Gamma)$ is not connected.   Let $\hat{W}^\pm =W^\pm \cup T_r(\Gamma)$.  Clearly, $\hat{W}^\pm$ are open sets with $\hat{W}^+\cap \hat{W}^-=T_r(\Gamma)$.  For each $x\in \Gamma$, $W^\pm \cap B_r(x)\neq \emptyset$ as $\Gamma=\partial W^\pm$. As the union of intersecting connected sets is connected, $W^\pm\cup B_r(x)$ is connected.  It readily follows that both $\hat{W}^-$ and $\hat{W}^+$ are connected.  Finally, by the Mayer-Vietoris long exact sequence for reduced homology, as $\Real^{n+1}=\hat{W}^+\cup \hat{W}^-$ is simply connected and both $\hat{W}^\pm$ are connected, $T_r(\Gamma)=\hat{W}^+\cap \hat{W}^-$ must be connected.  This contradicts our choice of $r$ and proves the lemma. 
	\end{proof}
Another elementary fact is that the level set flow remains connected up to and including its first disconnection time.
\begin{lem}\label{LevSetConnLem}
Let $\set{\Gamma_t}_{t\in [0,T]}$ be a level set flow of compact sets in $\Real^{n+1}$.  If $\Gamma_t$ is connected for $t\in [0, t_0)$, then $\Gamma_{t_0}$ is connected. 
\end{lem}
\begin{proof}
By the definition and basic properties of level set flow $\lim_{t\to t_0^-} \Gamma_t=\Gamma_{t_0}$ in Hausdorff distance.  Indeed, on the one hand, by the avoidance principle, 
$$\Gamma_{t_0}\subset T_{\sqrt{4n(t_0-t)}}(\Gamma_{t}).$$ 
On the other, as the space-time track of the level set flow, $\Real^{n+1}\times[0,T]\backslash W[0,T]$, is closed and $\Gamma_{t_0}$ is compact, for every $\epsilon>0$, there is a $\delta>0$ so that if $0<t_0-t<\delta$, then $\Gamma_t\subset T_{\epsilon}(\Gamma_{t_0})$.  
 Hence, if $\Gamma_{t_0}$ is disconnected, then for $t<t_0$ close enough to $t_0$, $\Gamma_t$ is disconnected, proving the claim.
\end{proof}
 The next result summarizes and extends \cite{BW4} and provides a description of the regularity properties of strong canonical boundary motions flows in $\Real^4$ of low entropy. 
\begin{prop} \label{BWSumProp}
Let $\left(E_0, E, \mathcal{K}=\set{\mu_t}_{t\geq 0}\right)$ be a strong canonical boundary motion in $\Real^4$.  Suppose the flow has extinction time $T$ and $\Sigma_0=\partial E_0$  satisfies $\lambda(\Sigma_0)< \Lambda_2$. 
\begin{enumerate}
\item For each $t\in [0,T)$, there are a finite, possibly empty, set of points $x_1,\ldots, x_{m(t)}\in\Real^4$ so that $\mu_t=\mathcal{H}^3\measrestrict \Sigma_t$ where $\Sigma_t$ is a hypersurface in $\Real^4\backslash \set{x_1, \ldots, x_m}$.
\item For an open dense subset $I\subset [0,T]$, if $t\in I$, then $\mu_t=\mathcal{H}^3\measrestrict \Sigma_t$ where $\Sigma_t$ is a closed hypersurface. That is, $m(t)=0$.
\item Let $(x_0,t_0)\in \Real^4\times(0,T]$ be a point at which $\mathcal{K}$ has positive Gaussian density. If $\set{\nu_t}_{t\in \Real}=\mathcal{T}\in \mathrm{Tan}_{(x_0,t_0)}\mathcal{K}$, then $\nu_{-1}=\mathcal{H}^3 \measrestrict \Upsilon$ where $\Upsilon$ is a smooth self-shrinker and either $\Upsilon$ is closed or it is asymptotically conical.  Moreover, whichever holds depends only on $(x_0,t_0)$ and not on the choice of tangent flow.
\item For each $(x_0,t_0)\in \Real^4\times(0,T]$ for which $\mathrm{Tan}_{(x_0,t_0)}\mathcal{K}$ contains an asymptotically conical shrinker, there is an $R_0=R_0(x_0,t_0,\partial E_0)>0$ so that for all $R\in (0,R_0]$
$$\Sigma_{t_0}(x_0,R)=\spt(\mu_{t_0} )\cap B_{R}^*(x_0)= \Sigma_{t_0}\cap  B_{R}^*(x_0) =\partial E_{t_0} \cap  B_{R}^*(x_0)=\partial^* E_{t_0} \cap  B_{R}^*(x_0),
$$
is a connected hypersurface that divides $B_R^*(x_0)$ into two components, one contained in $E_t$ and one disjoint from it.  Here $ B_{R}^*(x_0)= B_{R}(x_0)\backslash \set{x_0}$.
\end{enumerate}
\end{prop}
\begin{proof} Note first that as $(E_0,E, \mathcal{K})$ is a strong canonical boundary motion, $(E,\mathcal{K})$ is a canonical boundary motion in the sense of \cite{BW4} -- see Theorem 2.3 and the discussion at the beginning of Section 4 of \cite{BW4}. 
As such, Items (1) and (2) are both immediate consequences of \cite[Theorem 4.3]{BW4} -- see \cite[Corollary 4.4]{BW4} and the proof of \cite[Theorem 4.5]{BW4} for details.  Item (3) follows from \cite[Proposition 4.1 and Lemma 4.2]{BW4}.

It remains to show Item (4).  First, set $\epsilon_0=\Lambda_2-\lambda(\partial E_0)>0$. Next observe that if $(x_0,t_0)$ is a singular point of $\mathcal{K}$, then, by hypothesis, it is a non-compact singularity and so, by
\cite[Theorem 4.2(2)]{BW4}, there is a $\alpha=\alpha(\epsilon_0)>0$ and a $\rho_0=\rho_0(x_0,t_0)>0$ so that for all $(\rho,t)\in (0,\rho_0)\times (t_0-\rho^2, t_0+\rho^2)$,
$$
A_t(x_0,t_0,\rho)=\Sigma_t \cap\left( B_{2\alpha \rho}(x_0)\backslash \bar{B}_{\frac{1}{2}\alpha \rho}(x_0)\right)=\spt(\mu_t)\cap \left(B_{2\alpha \rho}(x_0)\backslash \bar{B}_{\frac{1}{2}\alpha \rho}(x_0)\right)
$$
is a connected non-empty hypersurface that is proper in $B_{2\alpha \rho}(x_0)\backslash \bar{B}_{\frac{1}{2}\alpha \rho}(x_0)$. The same is true if $(x_0,t_0)$ is not a singular point as then $\mathrm{Tan}_{(x_0,t_0)}\mathcal{K}$ consists of a static hyperplane.  
For $\rho\in (0,\rho_0)$, let
$$
A(x_0,t_0,\rho)=\bigcup_{t\in(t_0-\rho^2, t_0+\rho^2) }A_t(x_0,t_0,\rho)\times\set{t}\subset \Real^4\times \Real=\Real^5
$$
this is a connected non-empty hypersurface  that is proper in the hollow space-time cylinder
$$
C(x_0,t_0,\rho)=\left(B_{2\alpha \rho}(x_0)\backslash \bar{B}_{\frac{1}{2}\alpha \rho}(x_0)\right) \times (t_0-\rho^2, t_0+\rho^2 ).
$$
Clearly,
$A_t(x_t,t_0,\rho)\times \set{t}=A(x_0,t_0,\rho)\cap \left(\Real^4\times \set{t}\right)$ and this intersection is transverse.

By Item (3) of the definition of canonical boundary motion, $\spt(\mu_{t})=\overline{\partial^*E_{t}}$, and so
$$
A(x_0,t_0,\rho)=\overline{\partial^*E}\cap C(x_0,t_0,\rho).
$$
As $A(x_0,t_0,\rho)$ is smooth, every point is in the reduced boundary and so
$$
A(x_0,t_0,\rho)={\partial^*E}\cap C(x_0,t_0,\rho).
$$
Hence, by Item (4) of the definition of a strong canonical boundary motion,
$$
A(x_0,t_0,\rho)={\partial^*E}\cap C(x_0,t_0,\rho)=\overline{\partial^*E}\cap C(x_0,t_0,\rho)=\partial E \cap C(x_0,t_0,\rho).
$$
Together with the fact that that $A(x_0,t_0,\rho)$ meets $ \Real^4\times \set{t_0}$ transversally, this means
$$
A_{t_0}(x_0,t_0,\rho)=\partial^* E_{t_0} \cap \left(B_{2\alpha \rho}(x_0)\backslash \bar{B}_{\frac{1}{2}\alpha \rho}(x_0)\right)=\partial E_{t_0} \cap\left(B_{2\alpha \rho}(x_0)\backslash \bar{B}_{\frac{1}{2}\alpha \rho}(x_0)\right).
$$

Set $R_0=2\alpha \rho_0$ and, for any $R\in (0,R_0)$, let
$$
\Sigma_{t_0}(x_0,R)=\bigcup_{i=0}^\infty A_{t_0}(x_0,t_0,2^{-i} R).
$$
By the above, $\Sigma_{t_0}(x_0,R)$ is a connected non-empty hypersurface proper in  $B_{R}^*(x_0)$ and, moreover,
$$
\Sigma_{t_0}(x_0,R)=\partial^* E_{t_0} \cap  B_{R}^*(x_0)= \partial E_{t_0} \cap  B_{R}^*(x_0) =\spt(\mu_{t_0} )\cap B_{R}^*(x_0)= \Sigma_{t_0}\cap  B_{R}^*(x_0).
$$
Finally, as $\Sigma_{t_0}(x_0,R)$ is connected, non-empty and proper in $ B_{R}^*(x_0)$, $B_{R}^*(x_0)\backslash \Sigma_{t_0}(x_0,R)$ has two components.  On the one hand, $\Sigma_{t_0}(x_0,R)\subset \partial E_{t_0}$ implies at least one of these is a subset of $E_t$.  On the other, $\Sigma_{t_0}(x_0,R)\subset \partial^* E_{t_0}$ means the other is disjoint from $E_{t_0}$.
\end{proof}

%
%
%
%
%

Next we use the above regularity properties to relate the level set flow and its interior for strong canonical boundary motions of low entropy -- compare with Remark \ref{StrongBndryRem}.
\begin{prop}\label{StrongBdryMotProp}
Let $(E_0, E, \mathcal{K}=\set{\mu_t}_{t\geq 0})$ be a strong canonical boundary motion in $\Real^4$ with $\lambda(\partial E_0)<\Lambda_2$  and let $\set{\Gamma_t}_{t\in [0,T]}$ be the level set flow with $\Gamma_0=\partial E_0$. For any $s\in (0, T]$ there are a finite, possible empty, set of isolated points of $\Gamma_s$,  $p_1, \ldots, p_{M(s)}$, so that $\mathcal{K}$ has a closed singularity at $(p_i,s)$ and
$$
\Gamma_s\setminus \set{p_1, \ldots, p_{M(s)}}=\spt(\mu_s)=\partial {E}_s\setminus \set{p_1, \ldots, p_{M(s)}}=\partial (\Real^4\backslash \bar{E}_s).
$$ 
If, in addition, $\Gamma_s$ is connected and not a point, then $E_s=W^+[s]$ and $\Gamma_s=\partial W^\pm[s].$
\end{prop}
\begin{proof}
 As the level set flow is the biggest flow,  $\spt(\mu_t)\subset \Gamma_t$ -- see \cite[10.7]{I1}.  Pick a $s\in (0, T]$,  the entropy assumption ensures that there are at most a finite set of points $p_1, \ldots, p_{M(s)}\in \Real^4$ so $\mathcal{K}$ has a closed singularity at $(p_i,s)$ -- see \cite[Theorem 4.3 and Corollary 4.4]{BW4}. Moreover, there are radii $r_i>0$ so that the Gaussian density of $\mathcal{K}$ at any $(p,s)$ with $p\in B_{r_i}(p_i)\setminus\set{p_i}$ is zero.  Hence, $p_i\not\in \spt(\mu_s)$ and so $\spt(\mu_s)\subset \Gamma_s\setminus\set{p_1, \ldots, p_{M(s)}}$.  
 	
Now pick a $x_0\in \Gamma_s\backslash \set{p_1, \ldots, p_{M(s)}}$.  Let $\mathcal{T}\in \mathrm{Tan}_{(x_0,s)}\mathcal{K}$ be a tangent flow to $\mathcal{K}$ at the point $(x_0,s)$. By Item (4) of the definition of strong canonical boundary motion,  $(x_0,s)\in \overline{\partial^* E}$.  Hence, there is a sequence $(x_i, s_i)\in \partial^* E$ with $s_i>0$ and $\lim_{i\to \infty} (x_i, s_i)= (x_0,s)$.  As $(x_i, s_i)\in \partial^* E$, the Gaussian density of $\mathcal{K}$ at $(x_i,s_i)$ is at least $1$ and so, by the upper semicontinuity property of Gaussian density, the Gaussian density of $\mathcal{K}$ at $(x_0,s)$ is positive and so $\mathcal{T}$ is non-trivial.  Hence, by Item (3) of Proposition \ref{BWSumProp} and the fact that $x_0\neq p_i$ for any $1\leq i \leq M(s)$, $\mathcal{T}=\set{\nu_t}_{t\in \Real}$ is asymptotically conical.  A further consequence is that the $p_i$ are isolated points of $\Gamma_s$.
 
Thus, Item (4) of Proposition \ref{BWSumProp} implies that there is a $R_0>0$ so for all $R\in (0,R_0)$,  $\spt(\mu_s)\cap B_{R}^*(x_0)$ is  non-trivial.  As $\spt(\mu_s)$ is closed, this means that $x_0\in \spt(\mu_s)$ and hence, $\spt(\mu_s)=\Gamma_t\setminus\set{p_1, \ldots, p_{M(s)}}$ for all $s\in (0,T]$ proving the first equality. To see the second equality, first note that, by definition, $\partial E_s\subset \Gamma_s$.  Now suppose that $x_0\in \Gamma_s\setminus\set{p_1, \ldots, p_{M(s)}}$. By what we have already shown, $x_0\in \spt(\mu_s)$ and Item (4) of Proposition \ref{BWSumProp} both hold at $(x_0,s)$.  Hence, there is a $R_0>0$ so for all $R\in (0,R_0)$, 
$$
\spt(\mu_s)\cap B_{R}^*(x_0)=\partial E_s\cap B_{R}^*(x_0)
$$ 
and this intersection is non-empty.  As the topological boundary of a set is closed, $x_0\in \partial E_s$ and so $\Gamma_s\setminus\set{p_1, \ldots, p_{M(s)}}\subset \partial E_s$, completing the proof of the second equality. As $\spt(\mu_s)=\partial E_s \setminus \set{p_1,\ldots, p_{M(s)}}$,  Item (4) of Proposition \ref{BWSumProp} and the above argument implies that $\partial E_s \setminus \set{p_1, \ldots, p_{M(s)}}\subset \partial(\Real^4\setminus \bar{E}_s)$. Clearly, any $p_i\in \bar{E}_s$ is an interior point and so $p_i\not\in \partial \bar{E}_s$. Hence, as $\partial(\Real^4\setminus \bar{E}_s)=\partial \bar{E}_s \subset \partial E_s$, the third equality follows.

To complete the proof, first observe that, if $\Gamma_s$ is connected and not a single point, then $M(s)=0$ -- i.e., there are no closed singularities at time $s$ and $\Gamma_s=\partial E_s$.  By definition, $E_s\subset W^+[s]$ and $\partial E_s\subset \partial W^+[s]\subset \Gamma_s$.  As $\partial E_s=\Gamma_s$, this immediately implies $\Gamma_s=\partial W^+[s]$.    Similarly, by definition $\partial W^-[s]\subset \Gamma_s$ and so, for any $x\in \partial W^-[s]$,  Item (4) of Proposition \ref{BWSumProp} implies that there is an $R>0$ so that  $B_R^*(x)\cap \Gamma_s$ divides $B_R^*(x)$ into exactly two components, $U^\pm(x)$, with
$\partial U^\pm (x)\cap B_R^*(x)= \Gamma_{s}\cap B_R^*(x)$ .  

Moreover,  up to relabeling, $U^+(x)\subset E_s$ and $U^-(x)\cap E_s=\emptyset$.  As $x\in \partial W^-[s]$ and $W^-[s]\cap E_s=\emptyset$, $U^-(x)\subset W^-[s]$ and so $\partial W^-[s]\cap B_R^*(x)=\Gamma_s\cap B_R^*(x)$.  Hence, as $x\in \partial W^-[s]\subset \Gamma_{s}$, $B_R(x)\cap \Gamma_{s}\subset \partial W^-[s]$ and so $\partial W^-[s]$  is an open non-empty subset of $\Gamma_{s}$.  As $\partial W^-[s]$ is also closed and $\Gamma_{s}$ is assumed to be connected, $\Gamma_s= \partial W^-[s]$.  

Finally, let $\Omega=W^+[s]\backslash E_s$.  As $\partial E_s=\Gamma_s=\partial W^+[s]$, $\partial \Omega\subset \Gamma_s$.  For each $x\in \Gamma_s$, Item (4) of Proposition \ref{BWSumProp}, implies that, for $R$ sufficiently small, $B_R(x)\backslash \Gamma_s$ consists of two components one disjoint from $E_s$ and one contained in $E_s$.  As $B_R(x)\cap W^-[s]\neq \emptyset$ the component disjoint from $E_s$ is contained in $W^-[s]$ and so is disjoint from $\Omega$.   Likewise, the component contained in $E_s$ is disjoint from $\Omega$ by construction.  Hence, $\Omega\cap B_R(x)=\emptyset$ and so $x\not\in \partial \Omega$.  As $x$ was arbitrary, this means $\partial \Omega=\emptyset$ which implies $\Omega=\emptyset$.  That is,  $E_s=W^+[s]$.
\end{proof}

We use the preceding results and ideas from \cite{W4} to show that strong canonical boundary motions remain connected until they disappear.  That is, we show Theorem \ref{main} for strong canonical boundary motions. 
\begin{prop}\label{leftcontstrong}
Let $(E_0, E, \mathcal{K}=\set{\mu_t}_{t\geq 0})$ be a strong canonical boundary motion in $\Real^4$ with $\partial E_0$ connected and $\lambda[\partial E_0]<\Lambda_2$.
If $\set{\Gamma_t}_{t\in [0,T]}$ is the level set flow with $\Gamma_0=\partial E_0$ and extinction time $T$, then $\Gamma_t$ is connected and $n(t)=2$ for all $t\in [0,T)$.  \end{prop}
\begin{proof} 
As $\partial E_0$ is connected, bounded and $\partial E_0=\Sigma$ is compact, $W^+[0]=E_0$.  As $\Sigma$ is a connected hypersurface, there is a $\delta>0$ so that  $\Gamma_{t}$ is a smooth flow for $t\in [0,\delta]$ and so $\Gamma_t$ is connected, $n(t)=2$  and $W^+[t]=E_t$ for $t\in [0,\delta]$.  Let
$$
t_{dis}=\sup\{t\in (0,T)|n(s)=2 \mbox{ and $\Gamma_s$ is connected for all $0\leq s<t$}\}
$$ 
be the first possible disconnection time.  Clearly, $t_{dis}>\delta$ and if $t_{dis}=T$, then we are done.  In what follows we suppose $t_{dis}<T$ and derive a contradiction.  

First,  observe that, by construction, $t_{dis}$ must be a singular time, but not the extinction time of the flow. As such, for any $(x,t_0)\in \Real^4\times (0, t_{dis}]$, for which $\mathcal{K}$ has positive Gaussian density all tangent flows to $\mathcal{K}$ at $(x,t_{0})$ are asymptotically conical. 
Indeed, by Proposition \ref{BWSumProp}, if a tangent flow at $(x, t_0)$ was closed, then, as $\Gamma_t$ was connected for $t<t_0\leq t_{dis}$, for $t<t_{0}$ and $t$ close enough to $t_{0}$, $\spt({\mu_t})$ would also be a closed connected hypersurface. This would imply that the whole flow becomes extinct at $t_{0}$, contradicting the fact that $t_{dis}<T$ is not the extinction time.

By Lemma \ref{LevSetConnLem} and the definition of $t_{dis}$, $\Gamma_{t}$ is connected for all $t\in [0, t_{dis}]$.  Hence,  by Proposition \ref{StrongBdryMotProp}, for all $t\in [0, t_{dis}]$, $\Gamma_t=\spt(\mu_t)=\partial W^\pm[t]$ and $W^+[t]=E_t$.  We conclude that $n(t_{dis})=2$.  Indeed, if $n(t_{dis})\geq 3$, then, as $W^-[t_{dis}]$ is connected, there is a component, $\Omega$, of $W^+[t_{dis}]$ so $\Omega'=W^+[t_{dis}]\backslash \Omega$ is non-empty.  As $E_{t_{dis}}=W^+[t_{dis}]=\Omega\cup \Omega'$,  $\Omega\cap \Omega'=\emptyset$ and $\Omega, \Omega'$ are both open, $\Gamma_{t_{dis}}=\partial E_{t_{dis}}=\partial\Omega\cup \partial \Omega'$.  Hence, as $\Gamma_{t_{dis}}$ is connected, there is an $x\in  \partial\Omega\cap \partial \Omega'$.  By Item (4) of Proposition \ref{BWSumProp}, there is an $R>0$ so that $B_R^*(x)\cap E_{t_{dis}}$ has exactly one non-empty component, namely, $B_{R}^*(x)\cap \Omega=B_R^*(x)\cap \Omega'$.  This contradicts $\Omega\cap \Omega'=\emptyset$ and implies $n(t_{dis})=2$.

We claim there is a $t_1\in (t_{dis}, T)$ so $n(t_1)>2$.  If not, then, for all $t\in (t_{dis}, T)$, $n(t)=2$ and there would be no compact singularities at time $t$ as otherwise the flow would become extinct at $t<T$. Moreover, $W^+[t]=E_t$ for all $t\in (t_{dis},T)$.  This is because there is always exactly one unbounded component, $W^-[t]$, and so $E_t$ would have to be the unique component of $W^+[t]$.  As there are no compact singularities in  $[0, T)$, Proposition \ref{StrongBdryMotProp} implies  $\Gamma_{t}=\partial E_t=\partial W^+[t]$ and $\Gamma_t=\partial (\Real^4\setminus \bar{E}_t)=\partial W^-[t]$ and so $\Gamma_t$ is connected by  Lemma \ref{TopLem}.  That is, $t_{dis}=T$ which contradicts our assumption.

For each $t\in [0,T]$, let $\mathcal{C}[t]$ be the set of components of $W[t]$.   By \cite[Theorem 5.2]{W4}, for any $0\leq t <s \leq T$, there is a well-defined map $\pi_{s, t}: \mathcal{C}[s]\to \mathcal{C}[t]$ given by $\pi_{s,t}(\Omega_s)=\Omega_t$ if and only if there is a time-like continuous path in $W[t, s]$, connecting a point in $\Omega_s\times\set{s}$ to a point in $\Omega_{t}\times \set{t}$. 
 As already observed, $n(t_1)>2$, while $n(t_{dis})=2$.  Hence, the pigeonhole principle implies that there are two distinct components $\Omega_1, \Omega_2\in \mathcal{C}[t_1]$ so that $\pi_{t_1, t_{dis}}(\Omega_1)=\pi_{t_1,t_{dis}}(\Omega_2)=\Omega_{0}\in \mathcal{C}[t_{dis}]$ As $n(t_{dis})=2$, either $\Omega_0=W^+[t_{dis}]=E_{t_{dis}}$ or $\Omega_0=W^{-}[t_{dis}]$.  In the former case, $\Omega_1, \Omega_2\subset E_{t_1}$ and in the latter $\Omega_1$ and $\Omega_2$ are both disjoint from $E_{t_{1}}$. 

 Pick $x_1\in \Omega_1$ and $x_2\in \Omega_2$.  By definition, $(x_1,t_1),(x_2,t_1)$ are each connected via time-like paths in $W[t_{dis}, t_1]$ to the same component, $\Omega_0$, of $W[t_{dis}]\times \{t_{dis}\}$. Label the two paths, $p_1(s),p_2(s)$,  so that $p_1(1)=(x_1,t_1),p_2(1)=(x_2,t_1)$. As $p_1(0),p_2(0)$ are in the same component of $W[t_{dis}]\times\set{t_{dis}}$, there is a path $p_3$ in $ W[t_{dis}]$ so that $(p_3(0),t_{dis})=p_1(0), (p_3(1), t_{dis})=p_2(0)$. 
By the avoidance principle, there is a universal constant $C>0$ so that if $B_r(y)\cap\Gamma_{t_{dis}}=\emptyset$, then $(y,t)\subset W[t]$ for any $t\in[t_{dis},t_{dis}+Cr^2]$.
As $p_3([0,1])$ is compact, we can choose $0<r_0<\mathrm{dist}(p_3[0,1],\Gamma_{t_{dis}})$.  Hence,
\begin{equation*}
p_3([0,1])\times[t_{dis},t_{dis}+Cr_0^2]\subset  W[t_{dis},t_{dis}+Cr_0^2]
\end{equation*}
As such, if $\delta_1= \min\set{\frac{t_1-t_{dis}}{2},Cr_0^2}$, then for any $t\in (t_{dis}, t_{dis}+\delta_1)$, $(x_1,t_1), (x_2,t_1)$ can also be connected via time-like paths in $W[t,t']$ to the same components of $W[t]$.  That is,  $\pi_{t_1, t}(\Omega_1)=\pi_{t_1,t}(\Omega_2)$.

Now let
$$
I=\set{s\in [t_{dis}, t_1]:  \pi_{t_1, s}(\Omega_1)\neq\pi_{t_1,s}(\Omega_2)}.
$$
Clearly, $t_1\in I$ and, by what we just established $[t_{dis}, t_{dis}+\delta_1)\cap I=\emptyset$.
Let $t_*=\inf(I)$.  So $t_{dis}+\delta_1\leq t_*\leq t_1$ and $t_*$ is a singular time of the flow. Moreover, by the openness argument of the previous paragraph, $t_*\in I$. 

Let $\Omega^*_1=\pi_{t_1, t_*}(\Omega_1)$ and $\Omega_2^*=\pi_{t_1,t_*}(\Omega_2)$ be distinct components of $W[t_*]$.  If  $\mathcal{C}^*=\pi_{t_{*}, t_{dis}}^{-1} (\Omega_0)$, then  $\Omega_1^*$ and $\Omega_2^*$ are elements of this set and all elements of $\mathcal{C}^*$ are either subsets of $E_{t_*}$ or all are disjoint from $E_{t_*}$.  In fact, as $n(t_{dis})=2$, either
$$
E_{t_*}=\bigcup_{\Omega^*\in\mathcal C^*} \Omega^* \; \; \mbox{ or }\;\; W[t_*]\setminus E_{t_*}=\bigcup_{\Omega^*\in\mathcal C^*} \Omega^*.
$$
As there is only one unbounded component of $W[t_*]$,  we may, by relabeling,  assume that $\Omega_1^*$ is bounded.
\begin{claim}\label{BdryIntersection}
There is a point $p\in \partial \Omega_1^*$ so that, for any $R>0$, there exists an element $\Omega^*\in \mathcal{C}^*$ distinct from $\Omega_1^*$ so that $B_{R}(p)\cap\Omega^*\neq\emptyset$. Observe, it is possible $\Omega^*\neq \Omega_2^*$.
\end{claim}
To prove the claim, we only need to prove for the case both $\Omega_1^*,\Omega_2^*\subset E_{t_*}$, the case that they are both disjoint from $E_{t_*}$ follows  from the same argument.  If the claim is false, then for any $p\in\partial\Omega_1^*$, there is $R_p>0$ such that for any $\Omega^*\in C^*, \Omega^*\neq \Omega_1^*$, one has $B_{R_p}(p)\cap\Omega^*=\emptyset$. As $\Omega_1^*$ is assumed bounded,  $\partial\Omega_1^*$ is compact, and so there is a uniform $R_0$ such that $\dist(\partial\Omega_1^*,\cup_{\Omega^*\in\mathcal C^*, \Omega^*\neq\Omega_1^*}\partial\Omega^*)>R_0>0$.  As $E_{t_*}=\bigcup_{\Omega^*\in\mathcal C^*} \Omega^*$,
\begin{align*}
Z_{R_0}=\left\{x:\frac{1}{4}R_0\leq \dist(x,\Omega_1^*)\leq \frac{3}{4}R_0\right\}\cap \bar{E}_{t_*}=\emptyset.
\end{align*}
Here $Z_{R_0}$ is compact. By Proposition \ref{StrongBdryMotProp}, as $\Gamma_{t_*}\setminus \bar{ E}_{t_*}$ consists of a finite set of isolated points, one may shrink $R_0$ so
\begin{align}\label{DisconnectAtt*}
Z_{R_0}\cap \left( \bar{E}_{t_*}\cup \Gamma_{t_*}\right) =\emptyset.
\end{align}
As we are considering a strong canonical boundary motion, this implies
$$
Z_{R_0}\times \set{t_*}\cap \bar{E}=\emptyset.
$$
Hence, as $\bar{E}$ is a closed set and $Z_{R_0}$ is a compact set, there is a $\delta_*>0$ so
\begin{equation}\label{ZNotIntersectE}
Z_{R_0}\times[t_*-\delta_*, t_*]\cap \bar{E}=\emptyset
\end{equation}
An immediate consequence of this is that, $\pi_{t_{*}, s}(\Omega^*_1)$ is disjoint from $\pi_{t_*, s}(\Omega^*)$ for any $\Omega^*\in \mathcal{C}^*$ not equal to $\Omega_1^*$ and  all $s\in [t_*-\delta_*,t_*]$. Indeed, otherwise there would be a continuous space-time curve connecting $\Omega^*_1$ and some distinct component, $\Omega^*$, of $\mathcal C^*$ that lies entirely in $E\cap W[t_*-\delta_*, t_*]$. However, such a curve would have to intersect $Z_{R_0}$, contradicting \eqref{ZNotIntersectE}.  Hence,  $\pi_{t_1, s}(\Omega_1)\neq\pi_{t_1,s}(\Omega_2)$ for all $ s\in[t_*-\delta_*,t_*]$, contradicting the definition of $t_*$.

To complete the proof, observe that, for the point $p$ given by the claim, one has that, for any small $R$, either: 
\begin{enumerate}
	\item $B_{R}^*(p)\backslash \partial E_{t^*}$ contains at least three components;
	\item $B_{R}^*(p)\backslash \partial E_{t^*}$ contains two components of $E_{t_*}$;
	\item $B_{R}^*(p)\backslash \partial E_{t^*}$ contains two components both disjoint from $E_{t_*}$.  
\end{enumerate}

In any case, take a tangent flow at $P=(p,t_*)$.  The point $P$ cannot be a compact singularity by the choice of $p$ and so the tangent flow is asymptotically conical.  By Item (4) of Proposition \ref{BWSumProp}, for small enough $R>0$, the ball $B_R^*(p)\setminus\partial E_{t^*}$ has only two connected components, one contained in $E_{t_*}$ and one disjoint from $E_{t_*}$, so none of the above three situations can happen. This contradiction completes the proof.

\end{proof}

\section{Proof of Theorem \ref{main}}
In this section, we will show Theorem \ref{main}.  In fact, we will show a stronger result from which Theorem \ref{main} is an immediate consequence.
\begin{thm}\label{RefinedMainThm}
Let $\Sigma$ be a smooth closed connected hypersurface in $\Real^4$ with $\lambda[\Sigma]\leq \Lambda_2$.  If $\set{\Gamma_t}_{t\in [0,T]}$ is the level set flow with $\Gamma_0=\Sigma$ and extinction time $T$, then, for all $t\in [0,T]$,  $\Gamma_t$ is connected and $n(t)\leq 2$.  
Moreover, if 
$$E^+=W^+[0,T]\mbox{ and } E^-=W^-[0,T]\cup \left(\Real^4\times (T,\infty)\right),$$
then $E^\pm$ are both sets of locally finite perimeter in $\Real^4\times [0,\infty)$ and there are Brakke flows $\mathcal{K}^\pm$ so that
$$
(\tau^\pm=\pm \left(\partial [E^\pm]+[W^\pm[0]\times \set{0}]\right), \mathcal{K}^\pm)
$$
are both matching motions with initial condition $[\Sigma\times\set{0}]$.  Finally, 
$$
\overline{\partial^*E^\pm}=\partial E^\pm
$$
in $\Real^4\times(0,\infty)$. 
\end{thm}

\begin{proof}
First observe that we may assume $\lambda(\Sigma)<\Lambda_2$.  Indeed, suppose that $\lambda(\Sigma)=\Lambda_2$ and consider,  $\set{\Sigma_t}_{t\in [0, \delta]}$,  the classical solution to \eqref{MCF} with $\Sigma_0=\Sigma$.   As $\Sigma$ is closed, $\lambda(\Sigma)=F[\rho^{-1}( \Sigma-x)]$ for some $\rho>0$ and $x\in \Real^{n+1}$.  Hence, by the Huisken monotonicity formula, either $\lambda[\Sigma_{\delta}]<\Lambda_2$ or $\Sigma=\rho \Upsilon+x$ where $\Upsilon$ is a closed self-shrinker.  In the latter case, the theorem is immediate (as the flow will remain smooth until disappearing), while in the former, one can prove the result for $\Sigma_{\delta}$ and then use the fact that the flow was smooth to conclude it also for $\Sigma$.

As $\Sigma$ is a closed connected hypersurface in $\Real^4$, standard topological results, e.g., \cite{S}, imply that there is a connected bounded domain $E_0\subset \Real^4$ with $\partial E_0=\Sigma$.  Let $\mathbf{n}$ be the unit normal to $\Sigma$ that points into $E_0$.  As $\Sigma$ is smooth, there is an $\epsilon>0$ so for $|s|<\epsilon$
$$
\Sigma_s=\set{p+s\mathbf{n}(p)| p\in \Sigma} 
$$
is a foliation of ${T_{\epsilon}(\Sigma)}$ by hypersurfaces .  By shrinking $\epsilon$, if needed, we can also ensure that $\lambda(\Sigma_s)<\Lambda_2$ for $|s|< \epsilon$.
Pick a Lipschitz function $u_0:\Real^4\to \Real$ with the property that
\begin{enumerate}
\item $\set{u_0=s} =\Sigma_s$ for $|s|<\epsilon$,
\item $\set{u_0\leq -\epsilon}$ is the unbounded component of $\Real^4\backslash T_{\epsilon}(\Sigma)$; and 
\item $\set{u_0\geq \epsilon}$ is the bounded component of $\Real^4\backslash T_{\epsilon}(\Sigma)$.
\end{enumerate} 
Let $u$ be the solution to \ref{levelsetflow} with initial data $u_0$.  As such, if  $\Gamma_t^s=\set{x| u(t,x)=s}$, then for $|s|<\epsilon$, $\set{\Gamma_t^s}_{t\geq 0}$ is the level set flow with  $\Gamma_0^{s}=\Sigma_s$.  For each $i\geq 1$, pick $s_{\pm i} \in (-\epsilon, \epsilon)$ so that $s_{-i}<s_{-i-1}<0<s_{i+1}<s_{i}$ and $\lim_{i\to \pm \infty} s_i=0$.  Let $E_0^i=\set{u_0> s_i}$ and $E^{i}=\set{u> s_i}$.   By \cite[12.11]{I1}, one can choose the $s_i$ so that for $i\neq 0$, there are Brakke flows $\mathcal{K}^{i}$ so that $\left(E_0^{i}, E^{i}, \mathcal{K}^{i}\right)$ are all strong canonical boundary motion.

By Proposition \ref{leftcontstrong}, each $\Gamma_t^{i}=\Gamma^{s_i}_t=\set{u=s_i}$ is connected and for $t\in [0, T_i)$, where $T_i$ is the extinction time of the flow,  divides $\Real^4$ into two components $W^{\pm}_{i}[t]$ which satisfy $\Gamma_t^i=\partial W^\pm_i [t]$ and $W^+_i[t]=E_t^i=\set{x|u(t,x)>s_i}$.
Consider the open sets
$$
U^+[t]=\bigcup_{i=1}^\infty W^{+}_{i}[t]=\set{x| u(x,t)>0} \mbox{ and } U^-[t]=\bigcup_{i=1}^\infty W^{-}_{-i}[t]=\set{x| u(x,t)<0}.
$$
As each $W_i^\pm [t]$ is connected and $U^\pm[t]$ is their nested union, it follows that both the $U^\pm[t]$ are also connected. Moreover, as
$$
\Gamma_t=\set{x| u(x,t)=0} = \Real^4\backslash \left( U^+[t]\cup U^-[t]\right),
$$
 $W^\pm[t]=U^\pm[t]$. For $i\geq 1$ let,
$$
G_i[t]=\Real^4\backslash \left( W_{i}^+[t]\cup W_{-i}^-[t]\right)=\set{x| s_{-i}\leq u(x,t)\leq s_i} 
$$
and observe that each $G_i[t]$ is a compact set,  $G_{i+1}[t]\subset G_{i}[t]$ and $\bigcap_{i=1}^\infty G_i[t]=\Gamma_t$.
For $t\in [0,T]$, each $G_i[t]$ is connected.   Indeed,  $T_{-i}$,  the extinction time of $\set{\Gamma_t^{-i}}_{t\geq 0}$ must satisfy $T_{-i}>T$ and so, when $t\leq T$, $\Gamma_t^{-i}$ and $W_{-i}^\pm [t]$ are both non-empty and connected.  In particular, there is exactly one component,  $G_i^-[t]$, of $G_i[t]$ that contains $\Gamma_t^{-i}=\partial W^\pm_{-i}[t]$.  Let $G_i^+[t]=G_i[t]\backslash G_i^-[t]$, so $G_i^+[t]$ is closed and disjoint from $G^-_i[t]$.  Observe that $W_{-i}^-[t]\cup G_i^-[t]$ is a closed non-empty subset of $\overline{W_i^-[t]}=W_i^-[t]\cup \Gamma_t^i=\set{u\leq s_i}$  that is disjoint from $G_i^+[t]$.  As $G_i^+[t]$ is also a closed subset of $\overline{W_i^-[t]}$, $\overline{W_i^-[t]}=W_{-i}^-[t]\cup G_i^-[t]\cup G_i^+[t]$ and the closure of a connected set is connected,   $G_i^+[t]=\emptyset$, and so $G_i[t]$ is connected.  As the nested intersection of compact connected sets is connected, it follows that $\Gamma_t$ is connected and so we've proved the first part of the theorem.

To prove the second part of the theorem we observe that for $i\geq 1$,
$E^{i}=W_i^+[0,T]$ is a set of finite perimeter while
$$
F^{-i}=\set{u<s_{-i}}=\Real^4\times[0,\infty)\backslash \bar{E}^{-i}
=W_{-i}^-[0,T]\cup (\Real^{n+1}\times(T,\infty)),$$
is a set of locally finite perimeter.
 Moreover, there are matching motions 
 $$\left(\tau^{i}=\partial[E_i]+[W^{+}_{i}[0]], \mathcal{K}^{i}\right) \mbox{ and }\left(\tau^{-i}=-\left(\partial[F^{-i}]+[W^{-}_{-i}[0]]\right), \mathcal{K}^{-i}\right)$$
 with initial conditions $[\Sigma_{s_{\pm i}}\times \set{0}]$. As $\lambda(\Sigma_{s_{\pm i}})<\Lambda_2<2$,  \cite[Theorem 3.4]{W} implies that, up to passing to a subsequence, the two sequences of matching motions converge to matching motions $(\tau^+, \mathcal{K}^+)$ and $(\tau^-, \mathcal{K}^-)$ both with initial condition $[\Sigma\times\set{0}]$.  It further follows, from standard compactness results for sets of locally finite perimeter, that the $E^i$ converge,  as  sets of finite perimeter, to 
$$E^+=W^+[0,T]=\bigcup_{t\in [0,T]} U^+[t]=\set{u>0}$$
 which is also a set of finite perimeter.  Likewise, the $F^{-i}$ converge, as sets of locally finite perimeter, to  $F^-$ where
$$
F^-=W^-[0,T]\cup \left(\Real^4\times (T,\infty)\right)=\left(\bigcup_{t\in [0,T]} U^-[t]\right) \cup \left(\Real^4\times(T,\infty)\right)=\set{u<0}.
$$
Set $E^-=F^-$ and observe that $\tau^\pm =\pm \left(\partial [E^\pm]+[W^\pm[0]]\right)$ follows from 
the continuity of the boundary operator.  

It remains only to verify the claim about the reduced boundary.
To that end observe that in $\Real^4\times(0,\infty)$
$$
\overline{\partial^* E^+}\subset \partial E^+.
$$
We now suppose that $(x,t)\in  \partial E^+$ and $t>0$.
By definition, for any $r>0$, $B_r(x,t)\cap E^+\neq \emptyset$.  In particular, for $i$ sufficiently large $B_r(x,t)\cap W^+_{i}[0,T]\neq \emptyset$. As $x\in \Gamma_t$, we have $x\not\in  W^+_{i}[0,T]$ and so there is some point $(y_r,t_r)\in B_r(x,t)\cap \partial W^+_{i}[0,T]$. 
As $\left(E_0^{i}, E^i, \mathcal{K}^{i}\right)$ is a strong canonical boundary motion, it has only one compact singularity (at the terminal time $T_i<T$) and we can assume $t_r<T_i$. Hence, by Proposition \ref{StrongBdryMotProp} that $y_r\in \spt(\mu^{i}_{t_r})$ and so $(y_r,t_r)$ has positive Gaussian density for $\mathcal{K}^{i}$. 
 As $\mathcal{K}^{i}$ converges to $\mathcal{K}^+$, the upper semicontinuity of Gaussian density implies that $(x,t)$ is a point of positive Gaussian density for $\mathcal{K}^+$.  As $(\tau^+,\mathcal{K}^+)$ is a matching motion starting from $\Sigma$ and $\tau^+$ is the reduced boundary of a set of finite perimeter, $(x,t)\in \overline{\partial^* E^+}$. That is, $\overline{\partial^*E}=\partial E^+$ in $\Real^4\times(0,\infty)$.   Arguing in exactly the same way shows that $\overline{\partial^* E^-}=\partial E^-$ in $\Real^4\times(0,\infty)$.
\end{proof}

\begin{cor} \label{NonFatCor}
Let $\Sigma$ be a smooth closed connected hypersurface in $\Real^4$ with $\lambda[\Sigma]\leq \Lambda_2$.  If $\set{\Gamma_t}_{t\in [0,T]}$, the level set flow of $\Sigma$ with extinction time $T$, is non-fattening, then there is a unique strong canonical boundary motion $(E_0, E, \mathcal{K})$, with $\partial E_0=\Sigma$.
\end{cor}

\section{Forward clearing out}
In this section apply Theorem \ref{main} to prove Corollary \ref{forward}.  
\begin{proof}[Proof of Corollary \ref{forward}]
If the Corollary is not true, then there exist $C_i\rightarrow 0,\eta_i>0,R_i>0,0<\rho_i<\frac{R_i}{2C_i}$ satisfying $\frac{\eta_i}{C_i^3}\rightarrow 0$ and a sequence of non-fattening level set flows $\set{M_{i,t}}_{t\geq 0}$ with $M_{i,0}$,  closed hypersurfaces with $\lambda(M_{i,0})\leq \Lambda_2-\epsilon$, $M_{i,t}\neq\emptyset$ for $t\in(t_0,t_0+R_i^2)$ and so that the flows reach the space-time point $(x_0,t_0)$, but satisfy
\begin{equation*}
\mathcal H^3(B_{\rho_i}(x_0)\cap M_{t_0+C_i^2\rho_i^2})<\eta_i\rho_i^3.
\end{equation*}
By Theorem \ref{RefinedMainThm} and Corollary \ref{NonFatCor}, the $M_{i,t}$ agree with the slices of a strong canonical boundary motion $(E_{i,0}, E_i, \mathcal{K}_i=\set{\mu_{i,t}})$ for $t\in(t_0,t_0+R_i^2)$.  In particular, in this time interval, by Proposition \ref{StrongBdryMotProp},
$$
\mu_{i,t}=\mathcal{H}^3\measrestrict M_{i,t}
$$
and so $\mu_{i,t}(B_{\rho_i}(x_0)) <\eta_i\rho_i^3$.

Rescale the flows to get a new flow $\tilde{K}_i=D_{\frac{1}{C_i\rho_i}}(\mathcal{K}_i-(x_0,t_0))$ and let $\{\tilde M_{i,t}\}$ be the corresponding rescaling of the level set flow $\set{M_t}$. By Brakke's compactness theorem \cite[7.1]{I1}, up to passing to a subsequence,  $\tilde{K}_i$ converges to a limit flow $\tilde{K}=\{\tilde \mu_t\}$, and moreover,  by \cite[Theorem 3.5]{W}, $(T_i, \mathcal{K}_i)$ converge to a matching motion $(\tilde{T},\tilde{K})$.  Clearly, $\lambda[\tilde{\mu}_t] \leq \Lambda_2-\epsilon$.
We also have
\begin{equation*}
\tilde{\mu}_{i,1}\left(B_{\frac{1}{C_i}}(0)\right)<\frac{\eta_i}{(C_i)^3}\to 0
\end{equation*}
That is, $\tilde{\mu}_1(\Real^4)=0$ and so the limit flow $\tilde{K}$ must be extinct before $t=1$.  As $(\tilde{T},\tilde{K})$ is a matching motion, this means that $\tilde{K}$ must develop a collapsed singularity at some $t_e\leq 1$.  The entropy bound and  the classification of singularities given in Proposition \ref{BWSumProp} imply that this singularity has compact support.  Hence, by Brakke's regularity theorem, for large enough $i$, the flow $\{\tilde M_{i,t}\}$ must develop a compact singularity at some time $\tilde t_i<2$, and hence  $\{M_{i,t}\}$ develops a compact singularity at some time $t_i<t_0+2C_i^2\rho_i^2<t_0+\frac{2R_i^2}{4}<t_0+R_i^2$. As $M_{i,t_0+R_i^2}\neq\emptyset$ and there is a compact singularity before the extinction time, the flow must disconnect before time $t_0+R_i^2$, contradicting Theorem \ref{main}.
\end{proof}
%

\end{document}